\documentclass[english]{amsart}%
\usepackage{amsfonts}
\usepackage{amsmath,color}
\usepackage{amssymb}
\usepackage{amsthm}
\usepackage{amscd}
\usepackage{babel}
\usepackage[latin1]{inputenc}
\usepackage{graphicx}
\usepackage{amsmath}%
\providecommand{\U}[1]{\protect\rule{.1in}{.1in}}
\newtheorem{teor}{Theorem}

\newtheorem{prop}{Proposition}

\newtheorem{lem}{Lemma}
\newtheorem{exa}{Example}
\theoremstyle{definition}
\newtheorem{defi}{Definition}
\newtheorem*{rem}{Remark}

\renewcommand{\subjclassname}{AMS \textup{2010} Mathematics Subject
Classification\ }

\begin{document}

\author{J.M. Grau Ribas}
\address{Departamento de Matemáticas, Universidad de Oviedo\\
Avda. Calvo Sotelo s/n, 33007 Oviedo, Spain}
\email{grau@uniovi.es}

\title{A New Proof and Extension of the Odds-Theorem}

\begin{abstract}
There are $n$ independent Bernoulli random variables $I_{k}$ with parameters
$p_{k}$ that are observed sequentially. We consider a generalization of the
Last-Success-Problem considering $w_{k}$ positive payments if the player
successfully predicts that the last "1" occurs in the variable $I_{k}$. We
establish the optimal strategy and the expected profit in similar terms to
the Odds-Theorem. The proof provided here is an alternative proof to
the one Bruss provides in his Odds-Theorem (case $w_{i}=1$) that is even
simpler  and more elementary than his proof.
\end{abstract}

\maketitle

\subjclassname{60G40, 62L15}

\section{Introduction}

The Last-Success-Problem is the problem of maximizing the probability of
stopping on the last success in a finite sequence of Bernoulli trials. The
framework is as follows. There are $n$ Bernoulli random variables which are
observed sequentially. The problem is to find a stopping rule to maximize
the probability of stopping on the last $"1"$. We restrict ourselves here to
the case in which the random variables are independent. This problem has
been studied by Hill and Krengel \cite{1992} and Hsiau and Yang \cite{2000}
for the case in which the random variables are independent and was simply
and elegantly solved by T.F. Bruss in \cite{BR1} with the following famous
result.

\begin{teor}
\label{odds} (Odds-Theorem, F.T. Bruss 2000). Let $I_{1},I_{2},...,I_{n}$ be
$n$ independent Bernoulli random variables with known $n$. We denote by ($%
i=1,...,n$) $p_{i}$, the parameter of $I_{i}$; i.e. ($p_{i}=P(I_{i}=1)$).
Let $q_{i}=1-p_{i}$ and $r_{i}=p_{i}/q_{i}$. We define the index

\begin{equation*}
\mathbf{s}=
\begin{cases}
\max\{1\leq k\leq n: \sum_{j=k}^n r_j \geq 1\}, & \text{if $\sum_{i=1}^n
r_i\geq 1$ }; \\
1, & \text{ otherwise }%
\end{cases}%
\end{equation*}

To maximize the probability of stopping on the last $"1"$ of the sequence,
it is optimal to stop on the first $"1"$ we encounter among the variables $%
I_{\mathbf{s}},I_{\mathbf{s}+1},...,I_{n}$.

The optimal win probability is given by

\begin{equation*}
\mathcal{V}(p_1,...,p_n):=\displaystyle{\ \left( \prod_{j=\mathbf{s} }^{n}q_j \right) } \displaystyle{%
\left(\sum_{i=\mathbf{s} }^{n} r_i \right)} 
\end{equation*}
\end{teor}

This theorem was extended by T. Ferguson (see \cite{fergu}) in several ways. First, considering an infinite number of
Bernoulli variables. Second, the payoff for not stopping is allowed to be
different from the payoff for stopping on a success that is not the last
success. Third, the Bernoulli variables are allowed to be dependent.

In this paper, we present a generalization of the Last-Success-Problem,
considering a positive payoff, $w_{k}$, if the player stops on the last
success and this occurs at the $k$-th event. We establish the optimal
strategy and the expected profit in similar terms to the Odds-Theorem. The
proof we provide constitutes an alternative proof to the one
provided by Bruss that is even simpler  and more elementary.

\section{Threshold strategies}

In this section, we shall show that, under certain conditions, the optimal
strategy is a threshold strategy. Dynamic programming provides the
probability of winning and the optimal strategy in a simple way. In what
follows, we shall take into account the following definitions.

\begin{defi}
\label{defi1}  Let us define the following functions.

\begin{itemize}
\item $\mathbb{E}_{{\scriptsize \mathtt{Stop}}}(k)$ is the expected profit
if we stop at the $k$-th event with $I_{k}=1$
\begin{equation*}
\mathbb{E}_{{\scriptsize \mathtt{Stop}}}(k):=w_{k}\prod_{i=k+1}^{n}(1-p_{i})
\end{equation*}%

\item $\mathbb{E}_{\mathtt{Keep}}(k)$ is the expected profit after observing
the $r$-th event and continuing (not stopping) in order to adopt the optimal
strategy later on. The dynamic program that defines it by recurrence is:
\begin{equation*}
\mathbb{E}_{\mathtt{Keep}}(n)=0
\end{equation*}%
\begin{equation*}
\mathbb{E}_{\mathtt{Keep}}(k)=p_{k+1}\cdot \max \left\{ \mathbb{E}_{%
{\scriptsize {\mathtt{Stop}}}}(k+1),\mathbb{E}_{\mathtt{Keep}}(k+1)\right\}
+(1-p_{k})\cdot \mathbb{E}_{\mathtt{Keep}}(k+1)
\end{equation*}
\end{itemize}
\end{defi}

\begin{prop}
With the above definitions, it is obvious that the following strategy is
optimal:

$\diamond $ \textsf{Stop if $I_{k}=1$ and $\mathbb{E}_{{\scriptsize {\mathtt{%
Stop}}}}(k)> \mathbb{E}_{{\scriptsize {\mathtt{Keep}}}}(k)$ and continue
{ }   otherwise}.

In addition, using this strategy, the expected profit is $\mathbb{E}_\mathtt{%
Keep}(0).$
\end{prop}

\begin{defi}
We denote by the stopping set the set of indices in which the decision to
stop is optimal if the corresponding event is successful. That is:
\begin{equation*}
\Upsilon _{n}:=\{k:\mathbb{E}_{{\scriptsize {\mathtt{Stop}}}}(k)> \mathbb{%
E}_{{\scriptsize {\mathtt{Keep}}}}(k)\}
\end{equation*}
\end{defi}

\begin{exa}
Let us consider 9 random Bernoulli variables with the following parameters, $%
p_{i}$, and payoffs, $w_{i}$:
\begin{equation*}
\{{p_{1}}=\frac{1}{6},{p_{2}}=\frac{1}{10},{p_{3}}=\frac{1}{12},{p_{4}}=%
\frac{1}{3},{p_{5}}=\frac{1}{12},{p_{6}}=\frac{1}{10},{p_{7}}=\frac{1}{5},{%
p_{8}}=\frac{1}{10},{p_{9}}=\frac{1}{12}\}
\end{equation*}

\begin{equation*}
\{ {w_1} = 7,{w_2} = 4,{w_3} = 9,  {w_4} = 10,{w_5} = 6,{w_6} = 3,  {w_7} =
9,{w_8} = 9,{w_9} = 1\}
\end{equation*}

The corresponding dynamic program returns:
\begin{equation*}
\mathsf{Expected Profit}=\mathbb{E}_{{\scriptsize {\mathtt{Keep}}}}(0)=\frac{%
6721}{2000}
\end{equation*}%
and the stopping set

\begin{equation*}
\mathsf{Stopping Set} = \{4,5,7,8,9\}
\end{equation*}
\end{exa}

\begin{defi}
If the stopping set has a single stopping island, $\Upsilon _{n}=\{k:k\geq
\mathbf{k}\}$, we shall say that the optimal strategy is a\emph{\ threshold
strategy} and, in this case, $\mathbf{k}:=\min \Upsilon _{n}$ is the \emph{\
optimal threshold}. In the terminology of Chow, Robbins and Siegmund (see
\cite{mono}), we also state that the problem is \emph{a monotone problem},
which is not the case in the aforementioned example.
\end{defi}

\begin{rem}
Note that, for the optimal threshold, we have that
\begin{equation*}
\mathbf{k}=\min \{k:\mathbb{E}_{{\scriptsize {\mathtt{Stop}}}}(k)>
\mathbb{E}_{{\scriptsize {\mathtt{Keep}}}}(k)\}=1+\max \{k:\mathbb{E}_{%
{\scriptsize {\mathtt{Stop}}}}(k)\leq\mathbb{E}_{{\scriptsize {\mathtt{Keep}}}%
}(k)\}
\end{equation*}
\end{rem}

The following two easy results characterize monotone problems.

\begin{prop}
\label{guay} The problem is monotone if and only if for all $0<k<n$
\begin{equation*}
\mathbb{E}_{{\scriptsize {\mathtt{Stop}}}}(k)>\mathbb{E}_{{\scriptsize {%
\mathtt{Keep}}}}(k) \Rightarrow \mathbb{E}_{{\scriptsize {\mathtt{Stop}}}%
}(k+1) > \mathbb{E}_{{\scriptsize {\mathtt{Keep}}}}(k+1)
\end{equation*}
\end{prop}

\begin{prop}
\label{guay1} The problem is monotone if and only if for all $0<k\leq n$
\begin{equation*}
\mathbb{E}_{{\scriptsize {\mathtt{Stop}}}}(k)-\mathbb{E}_{{\scriptsize {%
\mathtt{Keep}}}}(k)
\end{equation*}%
change sign at the most once.
\end{prop}

With the following result, we present a sufficient condition for the problem
to be monotone. In particular, when the payment function, $w_{k}$, is
non-decreasing, the problem is monotone.

\begin{prop}
\label{umbral} If $w_{k+1}\geq (1-p_{k+1})w_{k}$ for all $k\in \{1,...,n-1\}$%
, then the problem is monotone.
\end{prop}

\begin{proof}
If we take into account the associated dynamic program, we see that $\mathbb{%
E}_{{\scriptsize {\mathtt{Keep}}}}(k)$ is non-increasing
\begin{equation*}
\mathbb{E}_{\mathtt{Keep}}(k)=p_{k+1} \max\left \{ \mathbb{E}_{{\scriptsize {%
\mathtt{Stop}}}}(k+1),\mathbb{E}_{\mathtt{Keep}}(k+1)\right \} + (1-p_k)
\mathbb{E}_{\mathtt{Keep}}(k+1)\geq \mathbb{E}_{\mathtt{Keep}}(k+1)
\end{equation*}

On the other hand, $\mathbb{E}_{{\scriptsize {\mathtt{Stop}}}}$ is
non-decreasing since
\begin{equation*}
\frac{\mathbb{E}_{{\scriptsize {\mathtt{Stop}}}}(k+1)}{\mathbb{E}_{%
{\scriptsize {\mathtt{Stop}}}}(k)}=\frac{w_{k+1}\prod_{i=k+2}^{n}(1-p_{i})}{%
w_{k}\prod_{i=k+1}^{n}(1-p_{i})}=\frac{w_{k+1}}{w_{k}(1-p_{k+1})}\geq 1
\end{equation*}

As a consequence, given that $\mathbb{E}_{{\scriptsize {\mathtt{Stop}}}}$ is
non-decreasing and $\mathbb{E}_{{\scriptsize {\mathtt{Keep}}}}$ is
non-increasing,
\begin{equation*}
\mathbb{E}_{{\scriptsize {\mathtt{Stop}}}}(k)\geq \mathbb{E}_{{\scriptsize {%
\mathtt{Keep}}}}(k)\Rightarrow \mathbb{E}_{{\scriptsize {\mathtt{Stop}}}%
}(k+1)\geq \mathbb{E}_{{\scriptsize {\mathtt{Keep}}}}(k+1)
\end{equation*}%
and we are able to use Proposition \ref{guay}.
\end{proof}

With Proposition \ref{guay}, it became evident that for the problem to be
monotone, it is sufficient for $\mathbb{E}_{{\scriptsize \mathtt{Stop}}}(r)$
to be non-decreasing. However, this is not a necessary condition. Actually,
the problem is monotone if and only if the difference $\mathbb{E}_{%
{\scriptsize {\mathtt{Stop}}}}(k)-\mathbb{E}_{{\scriptsize {\mathtt{Keep}}}%
}(k)$ presents one change of sign at the most. However, the verification of
this   statement  presents difficulties as the dynamic program does
not allow us to know an explicit expression of $\mathbb{E}_{{\scriptsize
\mathtt{Keep}}}(r)$. We shall see how to overcome this difficulty below.

\begin{defi}
Let us denote by $\overline{\mathbb{E}}_{\mathtt{Keep}}(k)$ the expected
profit after observing the $k$-th event and continuing in order to stop on
the next success to be found.
\begin{equation*}
\overline{\mathbb{E}}_{\mathtt{Keep}}(k):=\sum_{i=k+1}^{n}\left(
\prod_{j=k+1}^{i-1}(1-p_{j})\right) \cdot p_{i}\cdot \mathbb{E}_{%
{\scriptsize {\mathtt{Stop}}}}(i)
\end{equation*}%
In other words, $\overline{\mathbb{E}}_{\mathtt{Keep}}(k)$ is the expected
profit using the strategy of stopping on the first success after the $k$-th
event.
\end{defi}

It is clear from the definition itself that $\overline{\mathbb{E}}_{\mathtt{%
Keep}}(k)\leq \mathbb{E}_{\mathtt{Keep}}(k).$

\begin{lem}
\label{lemlim} Let $r_0$ be such that $\mathbb{E}_{{\scriptsize \mathtt{Stop}%
}}(r)> \overline{\mathbb{E}}_{\mathtt{Keep}}(r)$ for every $r>r_0$. Then,
$\mathbb{E}_{{\scriptsize \mathtt{Stop}}}(r)> \mathbb{E}_{\mathtt{Keep}%
}(r)$ for every $r>r_0$.
\end{lem}

\begin{proof}
Given $r_{0}$, let us consider the set $S=\{r>r_{0}:\mathbb{E}_{{\scriptsize
\mathtt{Stop}}}(r)\leq\mathbb{E}_{\mathtt{Keep}}(r)\}$. It is necessary to
prove that $S=\emptyset $. Let us assume that $S$ is nonempty and let $%
r^{\prime }$ be its maximum. This means that $\mathbb{E}_{\mathtt{Stop}%
}(r^{\prime })\leq\mathbb{E}_{{\scriptsize {\mathtt{Keep}}}}(r^{\prime })$ and $%
\mathbb{E}_{\mathtt{Stop}}(r^{\prime })>\overline{\mathbb{E}}_{%
{\scriptsize {\mathtt{Keep}}}}(r^{\prime })$, while $\mathbb{E}_{\mathtt{Stop%
}}(\mathbf{r}^{\prime })> \mathbb{E}_{\mathtt{Keep}}(\mathbf{r}^{\prime })
$ for all $\mathbf{r}^{\prime }>r^{\prime }$; but this is a contradiction.
This is because if $\mathbb{E}_{\mathtt{Stop}}(\mathbf{r}^{\prime })>
\mathbb{E}_{\mathtt{Keep}}(\mathbf{r}^{\prime })$ for all $\mathbf{r}%
^{\prime }>r^{\prime }$, then $\mathbb{E}_{\mathtt{Keep}}(r^{\prime })=%
\overline{\mathbb{E}}_{\mathtt{Keep}}(r^{\prime })$.
\end{proof}

Using this lemma, it is possible to reformulate Proposition \ref{guay} and
Proposition \ref{guay1} in terms of $\overline{\mathbb{E}}_{\mathtt{Keep}}(r)
$, which we can know explicitly.

\begin{prop}
\label{dos} If for all $0<r<n$ the following is true
\begin{equation*}
\mathbb{E}_{{\scriptsize {\mathtt{Stop}}}}(r)> \overline{\mathbb{E}}_{%
\mathtt{Keep}}(r)\Rightarrow \mathbb{E}_{{\scriptsize {\mathtt{Stop}}}%
}(r+1)>\overline{\mathbb{E}}_{\mathtt{Keep}}(r+1)
\end{equation*}%
then the problem is monotone.
\end{prop}

\begin{proof}
Let $r_{0}$ be the minimum of the stopping set. $\mathbb{E}_{{\scriptsize {%
\mathtt{Stop}}}}(r_{0})> \overline{\mathbb{E}}_{\mathtt{Keep}}(r_{0})$
and using the hypothesis inductively, we have that $\mathbb{E}_{{\scriptsize
{\mathtt{Stop}}}}(r)> \overline{\mathbb{E}}_{\mathtt{Keep}}(r)$ for all $%
r\geq r_{0}$. We thus find ourselves within the conditions of Lemma \ref%
{lemlim} and hence $\mathbb{E}_{{\scriptsize {\mathtt{Stop}}}}(r)>
\mathbb{E}_{\mathtt{Keep}}(r)$ for all $r\geq r_{0}$.

\begin{prop}
\label{dos11} The problem is monotone if and only if for all $0<k\leq n$
\begin{equation*}
\mathbb{E}_{{\scriptsize {\mathtt{Stop}}}}(k)-\overline{\mathbb{E}}_{%
{\scriptsize {\mathtt{Keep}}}}(k)
\end{equation*}%
change sign at the most once.
\end{prop}
\end{proof}

\begin{prop}
If the problem is monotone and $\mathbf{k}$ is the optimal threshold, then

\begin{equation*}
\mathbb{E}_{{\scriptsize {\mathtt{Stop}}}}(r)> \mathbb{E}_{\mathtt{Keep}%
}(r) \Longleftrightarrow \mathbb{E}_{{\scriptsize {\mathtt{Stop}}}}(r)>
\overline{\mathbb{E}}_{\mathtt{Keep}}(r)
\end{equation*}
\begin{equation*}
\mathbb{E}_{\mathtt{Keep}}(r)=
\begin{cases}
\overline{\mathbb{E}}_{\mathtt{Keep}}(r), & \text{if $r \geq \mathbf{k}$ };
\\
\overline{\mathbb{E}}_{\mathtt{Keep}}(\mathbf{k}-1), & \text{ if $r<\mathbf{k%
}$ }%
\end{cases}
\end{equation*}
\end{prop}

\section{The extended Odds-Theorem}

\begin{teor}
Let $I_{1},I_{2},...,I_{n}$ be $n$ independent Bernoulli random variables
with parameter $p_{i}$. Let $w_{i}$ be real positive numbers that represent
the payments a player receives for indicating the last $"1"$ in the variable
$I_{i}$. We define the index (with auxiliary $w_{0}:=0$)
\begin{equation*}
\mathbf{s}=\max \left\{ k:\sum_{j=k}^{n}\frac{w_{j}\cdot p_{j}}{1-p_{j}}\geq
w_{k-1}\right\}
\end{equation*}%
If the problem is monotone, then $\mathbf{s}$ is the optimal threshold. That
is, to maximize the expected profit, it is optimal to stop on the first $"1"$
we encounter among the variables $I_{\mathbf{s}},...,I_{n}$. Furthermore,
with this strategy, the expected profit is:

\begin{equation*}
\mathbb{E}=
\begin{cases}
\left( \prod_{j=\mathbf{s} }^{n}(1-p_j) \right)\sum_{i=\mathbf{s} }^{n}
\frac{w_i\cdot p_i}{1-p_i} , & \text{if $p_\mathbf{s}<1$ }; \\
w_\mathbf{s} \cdot\prod_{j=\mathbf{s+1} }^{n}(1-p_j), & \text{ if $p_\mathbf{%
s}=1$ }%
\end{cases}
\end{equation*}
\end{teor}

\begin{proof}
Recall that the optimal threshold is
\begin{equation*}
\mathbf{k} = 1+\max \{k: \mathbb{E}_{\mathtt{Stop}}(k)\leq\mathbb{E}_{\mathtt{%
Keep}}(k)\}=\max \{k: \mathbb{E}_{\mathtt{Stop}}(k-1)\leq\mathbb{E}_{\mathtt{%
Keep}}(k-1)\}.
\end{equation*}

We shall first assume that $p_{\mathbf{k}}<1$ and hence $p_{k}<1$ for all $k>%
\mathbf{k}$. Bear in mind that if $p_{k}=1$ for some $k>\mathbf{k}$, then $%
\mathbb{E}_{\mathtt{Stop}}(\mathbf{k})=0$, which would be a contradiction.
We shall first prove that $\mathbf{s}=\mathbf{k}$.

\begin{equation*}
\mathbf{k}=\max \{k:\mathbb{E}_{\mathtt{Stop}}(k-1)\leq \sum_{i=k}^{n}\left(
\prod_{j=k}^{i-1}(1-p_{j})\right) \cdot p_{i}\cdot \mathbb{E}_{\mathtt{Stop}%
}(i)\}
\end{equation*}%
as
\begin{equation*}
\mathbb{E}_{\mathtt{Stop}}(i)=w_{i}\prod_{t=i+1}^{n}(1-p_{t})
\end{equation*}

\begin{equation*}
\mathbf{k} =\max \left\{k:w_{k-1} \prod^n_{t=k} (1-p_t) \leq \sum_{i=k }^{n}
\left( \prod_{j=k }^{i-1}(1-p_j) \right)\cdot p_i \cdot w_i \prod^n_{t=i+1}
(1-p_t) \right\}
\end{equation*}

\begin{equation*}
\mathbf{k} =\max \left\{k:w_{k-1} \prod^n_{t=k} (1-p_t) \leq \sum_{i=k }^{n}
\left( \prod_{j=k }^{n}(1-p_j) \right)\cdot \frac{p_i}{1-p_i} \cdot w_i
\right\}
\end{equation*}

\begin{equation*}
\mathbf{k} =\max \left\{k:w_{k-1} \leq \sum_{i=k }^{n} \frac{p_i}{1-p_i}
\cdot w_i \right\} =\mathbf{s}
\end{equation*}

As to the value of the expected profit, which is in fact $\overline{\mathbb{E%
}}_{\mathtt{Keep}}(\mathbf{s}-1)$, we have

\begin{equation*}
\overline{\mathbb{E}}_{\mathtt{Keep}}(\mathbf{s}-1) = \sum_{i=\mathbf{s}%
}^{n} \left(\left( \prod_{j=\mathbf{s}}^{i-1}(1-p_j) \right) p_i \cdot
\mathbb{E}_{{\scriptsize {\mathtt{Stop}}}}(i) \right)
\end{equation*}

\begin{equation*}
\overline{\mathbb{E}}_{\mathtt{Keep}}(\mathbf{s}-1)= \sum_{i=\mathbf{s}%
}^{n} \left(\left( \prod_{j=\mathbf{s}}^{i-1}(1-p_j) \right) p_i \cdot w_i
\prod^n_{t=i+1} (1-p_t) \right)
\end{equation*}

and,   carrying out the same operations  as before, we have that

\begin{equation*}
\mathbb{E}=\overline{\mathbb{E}}_{\mathtt{Keep}}(\mathbf{s}-1)=\left(
\prod_{j=\mathbf{s} }^{n}(1-p_j) \right)\sum_{i=\mathbf{s} }^{n} \frac{%
w_i\cdot p_i}{1-p_i}
\end{equation*}

If $p_{\mathbf{k}}=1$, the proof that $\mathbf{s}=\mathbf{k}$ is the same.
As for the expected profit, bearing in mind that we shall stop at the $%
\mathbf{s}$-th variable with probability $1$, then

\begin{equation*}
\mathbb{E}= \overline{\mathbb{E}}_{\mathtt{Keep}}(\mathbf{s}-1)=\mathbb{E}_{%
\mathtt{Stop}}(\mathbf{s})= w_\mathbf{s} \cdot \left( \prod_{j=\mathbf{k}+1
}^{n}(1-p_j) \right).
\end{equation*}
\end{proof}
 
The previous proposition has as its particular case the famous Odds-Theorem
(Theorem \ref{odds}) when considering $w_{i}=1$. The proof provided
here is even more elementary and simpler than that provided by Bruss. The
preparatory results cannot be said to be absolutely original in substance,
but they are so in terms of their elucidation and hence the paper may be
said to be fully self-contained.

\section{Some application examples}

\subsection{The Best-choice Duration Problem.}

Let us consider the secretary problem with a payment $w_{k}=(n-k+1)$ for
selecting the best secretary in the $k$-th interview. Within the context of
this paper, we have $n$ independent Bernoulli random variables with
parameters $p_{k}=1/k$ and payoffs $w_{k}$. It is not difficult (though not
straightforward) to see that the problem is monotone. In this case, its
proof requires using Proposition \ref{dos11}.

\begin{equation*}
\mathbf{s}_n=\max \left\{k: \sum_{j=k}^n \frac{w_j \cdot p_j}{1-p_j}\geq
w_{k-1}\right\}=\max \left\{k: \sum_{j=k}^n \frac{(n-j+1) \cdot \frac{1}{j}}{%
1-\frac{1}{j}}\geq n-j+2\right\}
\end{equation*}

\begin{equation*}
\mathbf{s}_n=\max \left\{k:\frac{2n-2k+3}{n}\leq \sum_{k-2}^{n-1} \frac{1}{i}%
\right\}
\end{equation*}

from which it is   can easily be seen  that $\mathbf{s}_{n}/n$ tends
to \emph{rumour's constant}, which is the solution to the equation $%
2-2\,x+\log (x)=0$
\begin{equation*}
\vartheta :=-\frac{1}{2}W(-2e^{-2})=0.203187869....\text{ (A106533 in OEIS)}
\end{equation*}

and the asymptotic expected profit is $\mathbb{E}_{n}\sim n\cdot \vartheta
(1-\vartheta )=n\cdot 0.161902...$

\begin{rem}
Ferguson et al. in \cite{dura}, within the context of the Best-choice
Duration Problem, consider a payoff of $(n-k+1)/n$ and find the above
asymptotic values erroneously approximated as 0.20388... and 0.1618....
\end{rem}

\begin{rem}
If we consider $w_{k}:=1-k/n$ and $p_{k}=1/k$, the problem is equivalent to
the secretary problem considering a cost of $1/n$ for each interview and a
payment of $1$ for success. The asymptotic values are the same as in the
example and can be calculated in another way in \cite{bayon}.
\end{rem}

\subsection{The Best-choice and Minimal Duration Problem}

To the best of our knowledge, there is no study in the literature of this
problem, which consists in considering in the secretary problem a payment
for success equal to the number of interviews carried out. In the terms of
this paper, we shall have $n$ independent Bernoulli random variables with
parameters $p_{k}=1/k$ and payoffs $w_{k}=k.$ In this case, it is clear that
the problem is monotone (optimal threshold strategy) as $w_{k}$ is
increasing.

\begin{equation*}
\mathbf{s}_{n}=\max \left\{ k:\sum_{j=k}^{n}\frac{w_{j}\cdot p_{j}}{1-p_{j}}%
\geq w_{k-1}\right\} =\max \left\{ k:\sum_{j=k}^{n}\frac{j\cdot \frac{1}{j}}{%
1-\frac{1}{j}}\geq k-1\right\}
\end{equation*}%
Denoting by $\mathbf{H}(k)$ the $k$-th harmonic number, we have

\begin{equation*}
\sum_{i=k}^n \frac{i\frac{1}{i}}{1-\frac{1}{i}}= 1 - k +n + \,  \left(
\mathbf{H }(n-1) -  \mathbf{H} (k-2) \right)=1 -k + n + \sum_{k-2}^{n-1}
\frac{1}{i}
\end{equation*}

\begin{equation*}
\mathbf{s}_n= \max \left\{k: 1 - k + n + \sum_{k-2}^{n-1} \frac{1}{i}\geq
k-1\right\}= \max \left\{k: \sum_{k-2}^{n-1} \frac{1}{i}\geq 2k- n-2\right\}
\end{equation*}

from which it can easily be seen  that $\mathbf{s}_{n}/n$ tends to $%
1/2$ and the asymptotic expected profit is
\begin{equation*}
\mathbb{E}_{n}\sim \frac{n}{4}
\end{equation*}

\subsection{n Bernoulli variables with the same parameter $p_k=1/n$ and $w_k
= k$}

Let us consider $n$ independent Bernoulli random variables with parameters $%
p_{k}=1/k$ and payoffs $w_{k}=k$. The problem is monotone, as $w_{k}$ is
increasing.

\begin{equation*}
\mathbf{s}_n=\max \left\{k: \sum_{j=k}^n \frac{\left( 1 - k + n \right) \,
\left( k + n \right) }{2\, \left( -1 + n \right) }\geq k-1\right\}
\end{equation*}

\begin{equation*}
\mathbf{s}_n=\left\lfloor\frac{3 - 2\,n + \sqrt{1 + 8\,n^2}} {2}%
\right\rfloor \approx \frac{3}{2} + \left( -1 +  \sqrt{2} \right) \,n
\end{equation*}

\begin{equation*}
\mathbb{E}_n \sim n \left( -1 + \sqrt{2} \right) \,  e^{-2 + \sqrt{2}}=n
\cdot0.230579...
\end{equation*}

\subsection{n Bernoulli variables with the same parameter $p_k=p$ and $w_k
=n- k+1$}

Let us consider $n$ independent Bernoulli random variables with parameters $%
p_{k}=p$ and payoffs $w_{k}=n-k+1$. In this case, the problem is monotone as
\begin{equation*}
\mathbb{E}_{\mathtt{Stop}}(k)=\left( 1-k+n\right) \,{\left( 1-p\right) }%
^{-k+n}
\end{equation*}

\begin{equation*}
\overline{\mathbb{E}}_{\mathtt{Keep}}(k)=\frac{\left( -1+k-n\right) \,\left(
k-n\right) \,{\left( 1-p\right) }^{-1-k+n}\,p}{2}
\end{equation*}%
and $\mathbb{E}_{\mathtt{Stop}}(k)-\overline{\mathbb{E}}_{\mathtt{Keep}}(k)$
change sign at the most once.

Making

\begin{equation*}
\Omega:=\left\{k: \sum_{j=k}^n \frac{-\left( \left( -2 + k - n \right) \,
\left( -1 + k - n \right) \,p \right) }{2\, \left( -1 + p \right) }\geq
n-k+2\right\}
\end{equation*}

\begin{equation*}
\mathbf{s}_n=
\begin{cases}
\max \Omega , & \text{if $\Omega \neq\emptyset$ }; \\
1, & \text{ if $\Omega = \emptyset$ }%
\end{cases}
\end{equation*}

\begin{equation*}
\mathbf{s}_n=
\begin{cases}
\left\lfloor3 + n - \frac{2}{p}\right\rfloor , & \text{if $n > \frac{%
2\,\left( 1 - p \right) } {p}$ }; \\
1, & \text{ if $n \leq \frac{2\,\left( 1 - p \right) } {p}$ }%
\end{cases}
\end{equation*}

\begin{equation*}
\mathbb{E}_n=
\begin{cases}
\frac{{\left( 1 - p \right) }^ {-3 + \lfloor 2/p \rfloor}\,p\, \left( -2 +
\lfloor 2/p \rfloor \right) \, \left( -1 + \lfloor 2/p \rfloor \right) }{2},
& \text{if $n > \frac{2\,\left( 1 - p \right) } {p}$ }; \\
\frac{n\,\left( 1 + n \right) \, {\left( 1 - p \right) }^ {-1 + n}\,p}{2}, &
\text{ if $n \leq \frac{2\,\left( 1 - p \right) } {p}$ }%
\end{cases}
\end{equation*}

As expected, the optimal threshold is less than $n-\lceil 1/p\rceil +2$,
which is the   value  we obtain when considering the
Last-Success-Problem with parameters $p_{i}=p$ (see \cite{yo} and \cite{mal}%
).

\end{document}